\documentclass[12pt, leqno]{amsart}
\usepackage[T1]{fontenc}
\usepackage{a4, amsmath}
\usepackage{mathtools}
\usepackage{amssymb}
\usepackage{amsthm, amscd, mathdots}
\usepackage{enumerate}
\usepackage{hyperref}
\usepackage{cleveref}
\usepackage{datetime}
\usepackage{xcolor}
\usepackage{csquotes}

\usepackage[top=1in, bottom=1in, left=1in, right=1in]{geometry}

\theoremstyle{definition}
\newtheorem{defi}{Definition}[section]

\newtheorem{eg}[defi]{Example}

\theoremstyle{plain}
\newtheorem{prop}[defi]{Proposition}

\newtheorem{thm}[defi]{Theorem}
\newtheorem*{thm*}{Theorem}
\newtheorem*{prop*}{Proposition}
\newtheorem{cor}[defi]{Corollary}

\newtheorem{quess}[defi]{Questions}
\newtheorem*{ques*}{Question}

\theoremstyle{remark}
\newtheorem{rem}[defi]{Remark}

\crefname{thm}{theorem}{theorems}
\crefname{lem}{lemma}{lemmata}
\crefname{prp}{proposition}{propositions}
\crefname{cor}{corollary}{corollaries}
\crefname{obs}{observation}{observations}

\newcommand{\mc}{\mathcal}
\newcommand{\mf}{\mathfrak}
\newcommand{\mbb}{\mathbb}

\newcommand{\cc}{\mathbb C}
\newcommand{\rr}{\mathbb R}
\newcommand{\qq}{\mbb Q}

\newcommand{\zz}{\mbb Z}

\DeclareMathOperator{\Tr}{Tr}

\makeatletter
\newcommand*{\house}[1]{%
	\mathord{%
		\mathpalette\@house{#1}%
	}%
}
\newcommand*{\@house}[2]{%
	\dimen@=\fontdimen8 %
	\ifx#1\scriptscriptstyle\scriptscriptfont
	\else\ifx#1\scriptstyle\scriptfont
	\else\textfont\fi\fi
	3 %
	\sbox0{%
		$#1%
		\vrule width\dimen@\relax
		\overline{%
			\kern2\dimen@
			\begingroup 
			#2%
			\endgroup
			\kern2\dimen@
		}%
		\vrule width\dimen@\relax
		\mathsurround=1.5\dimen@ 
		$%
	}%
	\ht0=\dimexpr\ht0-\dimen@\relax
	\dp0=\dimexpr\dp0+2\dimen@\relax
	\vbox{%
		\kern\dimen@ 
		\copy0 %
	}%
}

\title{Failures of integral Springer's Theorem}
\author{Nicolas Daans}
\author{Vítězslav Kala}
\author{Jakub Krásenský}
\author{Pavlo Yatsyna}

\address{Charles University, Faculty of Mathematics and Physics, Department of Algebra, Soko\-lov\-sk\' a 83, 186~75 Praha~8, Czech Republic, and Department of Mathematics, Faculty of Science, KU Leuven, Celestijnenlaan 200B, 3001 Heverlee, Belgium}
\email[N.~Daans]{nicolas.daans@kuleuven.be}
\address{Charles University, Faculty of Mathematics and Physics, Department of Algebra, Soko\-lov\-sk\' a 83, 186~75 Praha~8, Czech Republic}
\email[V.~Kala]{vitezslav.kala@matfyz.cuni.cz}
\email[P.~Yatsyna]{p.yatsyna@matfyz.cuni.cz}
\address{Czech Technical University in Prague, Faculty of Information Technology, Department of Applied Mathematics, Thákurova~9, 160~00 Praha~6, Czech Republic}
\email[J.~Krásenský]{jakub.krasensky@fit.cvut.cz}
\date{\today}

\begin{document}
\begin{abstract}
We discuss the phenomenon where an element in a number field is not integrally represented by a given positive definite quadratic form, but becomes integrally represented by this form over a totally real extension of odd degree.
We prove that this phenomenon happens infinitely often, and, conversely, establish finiteness results about the situation when the quadratic form is fixed.
\end{abstract}

\makeatletter
\@namedef{subjclassname@2020}{%
  \textup{2020} Mathematics Subject Classification}
\makeatother
\thanks{The research was supported by {Czech Science Foundation} grant 21-00420M (N. D., V. K.), and by {Charles University} programmes PRIMUS/24/SCI/010 (N. D., P. Y.) and UNCE/24/SCI/022 (P. Y.). \\
This is the accepted manuscript of the article published in Proc. Amer. Math. Soc. 153 (2025), pp.~2369-2379. \href{https://doi.org/10.1090/proc/17141}{https://doi.org/10.1090/proc/17141}}

\maketitle

\section{Introduction}
A famous theorem of Springer says that, given a quadratic form defined over a field $K$ and a field extension $L/K$ of odd degree, the form is isotropic over $K$ if and only if it is so over $L$, see, for example, \cite[Corollary 18.5]{ElmanKarpenkoMerkurjev} for a modern proof.
By standard (de)homogenisation arguments, this theorem can be rephrased alternatively as a statement about representation by quadratic forms:
\begin{thm*}[{Springer, 1952 \cite{Springer}}]
Let $K$ be a field, $Q \in K[X_1, \ldots, X_n]$ a quadratic form over $K$, $L/K$ a field extension of odd degree, $a \in K$.

If $a\in Q(L)$, then $a\in Q(K)$.
\end{thm*}

For a commutative ring $R$ and a quadratic form $Q \in R[X_1, \ldots, X_n]$,
we say that $a\in R$ is \emph{represented by $Q$ over $R$} if $a=Q(x_1,\dots,x_n)$ for some $x_1, \ldots, x_n \in R$, and we denote by $Q(R)$ the subset of $R$ of all represented elements.

In number theory, one is often concerned with \textit{integral} representations of quadratic forms, and it is natural to ask whether the analogue of the above theorem holds in this context.
More precisely, let $K$ be a number field (respectively a local field of characteristic $0$), and denote by $\mc{O}_K$ its ring of integers (respectively its discrete valuation ring).
We ask:
\begin{ques*}[``Integral Springer's Theorem for representations'', ISTR]
Let $K$ be a number field or a local field of characteristic $0$.
Let $Q \in \mc{O}_K[X_1, \ldots, X_n]$ be a quadratic form over $\mc{O}_K$, $L/K$ a field extension of odd degree, $a \in \mc{O}_K$.

If $a\in Q(\mc{O}_L)$, does it necessarily follow that also $a\in Q(\mc{O}_K)$?
\end{ques*}
A positive answer has been obtained in the case where $K$ is a local field, or $K$ is a number field and $Q$ is an indefinite quadratic form in at least $3$ variables; see \Cref{T:Springer-local} and \Cref{T:Springer-indefinite-global} below and the discussion that follows.
In particular, this applies whenever $K$ is not totally real.

This leaves just one case to consider: when $K$ is a totally real number field and $Q$ is a definite form.
We shall consider in particular the case where $Q$ is positive definite, i.e.~represents only totally positive elements of $K$.
Although some of the existing literature on this topic has correctly insinuated that the ISTR might fail in this case (i.e.~the above Question has a negative answer), see e.g.~\cite[Introduction]{HeSpringer} and \cite[Section 7]{XuSpringer}, to the best of the authors' knowledge, no concrete examples thereof have so far been provided or discussed in the literature, at least not when also the extension $L$ is assumed to be totally real.
This is the goal of the present note.

After introducing some basic concepts and notations, we start in \Cref{sect:general-observations} by briefly discussing the cases where a positive answer to ISTR is known for number fields and local fields.
These provide sufficient (local) conditions for ISTR to hold or fail in certain contexts, also for positive definite quadratic forms.
In \Cref{sec:finite}, we use these conditions and established facts from the theory of integral quadratic forms, to derive general finiteness results which bound the potential failure of ISTR for positive definite quadratic forms.
In particular, we show:

\begin{thm*}[see \Cref{C:finitely-many-degree-d-failures} and \Cref{P:finiteness-HKK}]
Let $K$ be a totally real number field, $d$ a positive integer, $n \geq 5$, and $Q \in \mc{O}_K[X_1, \ldots, X_n]$ be a positive definite quadratic form.
There exist only finitely many $a \in \mc{O}_K \setminus Q(\mc{O}_K)$ (up to multiplication by squares of units) and finitely many minimal totally real extensions $L/K$ of degree $2d+1$ such that
$a\in Q(\mc{O}_L)$.
\end{thm*}

Similar statements exist for positive definite quadratic forms in $3$ or $4$ variables under more technical assumptions, see \Cref{rem:rank-3-4}. Moreover, a version of this statement for a fixed element $a$ actually holds without the assumption that the degree of $L/K$ is odd, see \Cref{L:finitely-many-degree-d-failures}.

In \Cref{sect:examples} we provide several classes of examples: we show that ISTR fails for quadratic forms over $\zz$ of all ranks $\geq 2$, even when restricting to totally real cubic Galois extensions, and for ranks $\geq 4$ we provide an infinite parametrized family of examples:
\begin{thm*}[see \Cref{C:infinite-family-Springer-failure}]
Let $k \in \zz^+$, $k \geq 4$ and let $K = \qq(\omega)$ where $\omega$ is a root of the irreducible polynomial $T^3+ T^2 - 2T - 1$. Let $Q(X_1,X_2,X_3,X_4)= X_1^2 + X_2^2 + X_3^2 + (8k+5)X_4^2$.

Then there exists a natural number $n$ such that $n\in Q(\mc O_K)$, but $n\not\in Q(\zz)$.
\end{thm*}
On the other hand, we give an example showing that the local conditions which guarantee that ISTR holds, as established in \Cref{sect:general-observations}, are not necessary, see \Cref{T:15-representation}.

\subsection*{Acknowledgments} We thank Zilong He for a friendly email exchange when we were exploring the topic of this paper, and the referee for several helpful suggestions.

\section{Background}
\label{sect:general-observations}

For a number field $K$, we denote by $\mc{O}_K^+$ the subset of $\mc{O}_K$ of \textit{totally positive elements} (i.e. of those that are positive in all real embeddings; however, we impose no restrictions in complex embeddings), by $\mc{O}_K^\times$ the set of units of $\mc{O}_K$, and by $\mc{O}_K^{\times 2}$ the set of squares of units.
Note that $0 \not\in \mc{O}_K^+$.
In particular, we denote by $\zz^+$ the set of (non-zero) natural numbers.
For subsets $S_1, S_2 \subseteq K$ we write $S_1S_2 = \lbrace xy \mid x \in S_1, y \in S_2 \rbrace$.
For $x, y \in K$, we write $x \succ y$ or $y \prec x$ if $x - y$ is totally positive.

Given an integral domain $R$ with field of fractions $K$, by a \emph{quadratic $R$-lattice} we shall mean a pair $(\Lambda, Q)$ where $\Lambda$ is a finitely generated $R$-submodule of a finite-dimensional $K$-vector space, and $Q : \Lambda \to R$ is a quadratic map, i.e.~a map satisfying $Q(av) = a^2Q(v)$ for $a \in R$ and $v \in \Lambda$, and such that the associated map
$$ b_Q : \Lambda \times \Lambda \to R : (v, w) \mapsto Q(v+w) - Q(v) - Q(w)$$
is bilinear.

Given a quadratic lattice $(\Lambda, Q)$, the $K$-dimension of $K\Lambda$ is called the \emph{rank} of $(\Lambda, Q)$.
We call a quadratic lattice $(\Lambda, Q)$ \emph{degenerate} if there exists $v \in \Lambda$ such that $Q(v+w) = Q(w)$ for all $w \in \Lambda$, \emph{nondegenerate} otherwise.

When $(\Lambda, Q)$ is a quadratic lattice and $\Lambda$ is a free $R$-module, then we may identify $\Lambda$ with $R^n$ for some $n \in \zz^+$, and there exists a homogeneous quadratic polynomial $f \in R[X_1, \ldots, X_n]$ such that $Q\bigl((x_1, \ldots, x_n)\bigr) = f(x_1, \ldots, x_n)$ for all $(x_1, \ldots, x_n) \in R^n$.
In this case, we also call $Q$ a \emph{quadratic form}.
Note that all quadratic lattices over a field are free, so in this case, the study of quadratic lattices reduces to that of quadratic forms.

Given a quadratic lattice $(\Lambda, Q)$ and $a \in R$, we say that \emph{$Q$ represents $a$} if $a = Q(v)$ for some $v \in \Lambda$.
We denote by $Q(R)$ the set of elements of $R$ represented by $Q$. (It is more common to denote this by $Q(\Lambda)$, but we chose a more convenient notation for the purposes of this paper, as we are interested in extensions of scalars.)

Given an extension $S/R$ of integral domains, and a quadratic lattice $(\Lambda, Q)$ over $R$, there exists a unique quadratic map $Q_S : \Lambda \otimes_R S \to S$ that extends $Q$; we call $(\Lambda \otimes_R S, Q_S)$ the \emph{extension of scalars} of $(\Lambda, Q)$ from $R$ to $S$, and also denote $\Lambda_S = \Lambda \otimes_R S$.
We might simply write $Q(S)$ instead of $Q_S(S)$.
Clearly, we have $Q(R) \subseteq Q(S) \cap R$, and in general, the inclusion might be strict.
Springer's Theorem implies that, if $L/K$ is a field extension of odd degree and $Q$ a quadratic form over $K$, then $Q(K) = Q(L) \cap K$.

Given a number field $K$, 
we call a quadratic lattice $(\Lambda, Q)$ over $\mc{O}_K$ \textit{indefinite}  if $K$ is not totally real or there is an embedding $\sigma:K\rightarrow \rr$ such that $\sigma(Q(\mc{O}_K))\subset\rr$ contains both positive and negative elements.
If $K$ is totally real, we call $\Lambda$ \emph{positive definite} if $Q(v)$ is totally positive for all $v \in \Lambda \setminus \lbrace 0 \rbrace$. 
Note that a positive definite quadratic lattice is in particular nondegenerate.

For a discrete prime $\mf{p}$ of $K$, denote by $\mc{O}_{\mf{p}}$ the completion of $\mc{O}_K$ with respect to $\mf{p}$. When $\mf{p}$ is archimedean, then $\mc{O}_{\mf{p}}$ is the completion of $K$ with respect to $\mf{p}$, i.e. it is $\rr$ or $\cc$.
If $(\Lambda, Q)$ is a quadratic $\mc{O}_K$-lattice and $a \in \mc{O}_K$, we say that $a$ is \emph{locally represented by $Q$ over $\mc{O}_K$} if $a \in Q(\mc{O}_\mf{p})$ for all primes $\mf{p}$ of $K$ (both discrete and archimedean).
In general, this is not sufficient to conclude that $a \in Q(\mc{O}_K)$, but in some cases it is; a quadratic $\mc{O}_K$-lattice $(\Lambda, Q)$ such that every element of $\mc{O}_K$ which is locally represented by $Q$ lies in $Q(\mc{O}_K)$, is called \emph{regular}. 

\smallskip

We now discuss briefly the known results regarding ISTR which are relevant to our setup, without attempting to give a complete account of the history or of related results in the literature.

\begin{thm}[Local Integral Springer's Theorem for representations, \cite{XuSpringer, HeSpringer}]\label{T:Springer-local}
Let $K$ be a local field of characteristic $0$ and $L/K$ be a field extension of odd degree.
Let $(\Lambda, Q)$ be a quadratic $\mc{O}_K$-lattice.
Then $Q(\mc{O}_K) = Q({\mc{O}_L}) \cap \mc{O}_K$.
\end{thm}
The above was first observed in the case where $K$ is either non-dyadic or $2$ is unramified in $K$ in \cite{XuSpringer}, using the local integral quadratic form theory developed by Riehm and O'Meara.
The missing case where $K$ is dyadic and $2$ is ramified was very recently solved in \cite{HeSpringer}.

The following globalization of \Cref{T:Springer-local} was proven for the case $K = \qq$ in \cite{XuSpringer} and also conjectured there to hold in general; the general case was then proven recently in \cite{HeSpringer} using the theory of base of norm generators (``BONGs'') as developed by Beli in among others~\cite{Beli03,Bel06,Beli10,Beli22}.
The proofs rely on a combination of local \Cref{T:Springer-local} with delicate computations of integral spinor norms.
\begin{thm}[Integral Springer's Theorem for representations by indefinite forms, \cite{XuSpringer, HeSpringer}]\label{T:Springer-indefinite-global}
Let $K$ be a number field and $L/K$ a field extension of odd degree.
Let $(\Lambda, Q)$ be an indefinite quadratic $\mc{O}_K$-lattice of rank at least $3$.
Then $Q({\mc{O}_K}) = Q({\mc{O}_L}) \cap \mc{O}_K$.
\end{thm}
As pointed out in \cite[Example 5.7]{HeSpringer}, the above result fails for lattices of rank $2$.
We note that the results of \cite{XuSpringer,HeSpringer} are actually more general than what we address in the above two theorems, since they consider not just representation of elements by quadratic lattices, but also representations of quadratic lattices by other quadratic lattices.
In fact, for quadratic lattices of rank at least $4$, \Cref{T:Springer-indefinite-global} can be deduced directly from  \Cref{T:Springer-local} via the following well-known local-global principle; see, for example, \cite{Hsia-Representation-indefinite} for an algebraic proof.

\begin{thm}[Local-global principle for integral representation by indefinite forms]\label{T:local-global-indefinite}
Let $K$ be a number field, $(\Lambda, Q)$ a nondegenerate indefinite quadratic $\mc{O}_K$-lattice of rank at least $4$, $a \in \mc{O}_K$.
Then $a \in Q({\mc{O}_K})$ if and only if $a$ is locally represented by $Q$ over $\mc{O}_K$.
\end{thm}

We now turn our attention to the positive definite case.
We shall provide evidence in \Cref{sect:examples} via several examples that no obvious version of \Cref{T:Springer-indefinite-global} (ISTR) can hold when replacing indefinite quadratic lattices by positive definite quadratic lattices, and in fact we suspect that failure of ISTR for positive definite quadratic forms happens often.

These results on ISTR should be viewed alongside the \textit{lifting problem} for universal quadratic forms. Briefly, for a number field $K$, a quadratic $\mc{O}_K$-lattice is \textit{universal} if it represents all elements of $\mc{O}_K^+$;
the existence and properties of universal lattices have been widely studied, see e.g. the surveys \cite{VK,Ksur}.
The lifting problem then asks if it is possible for an $\mc{O}_K$-lattice to be universal over a larger field $L\supset K$.
Analogously to our expectations concerning ISTR,
when the lattice is indefinite, then a version of local-global principle holds and the extension of scalars is universal quite often \cite{HHX,XZ}.
However, in the positive definite case, this seems to happen very rarely \cite{KL,KY2,KY23,KY4}.

\bigskip

We conclude this section by pointing out that ISTR can only fail for elements which are locally represented over the base field. 

\begin{prop}\label{O:only-local-failures}
Let $L/K$ be an odd-degree extension of number fields, $(\Lambda,Q)$ a quadratic $\mc{O}_K$-lattice. If $a \in Q(\mc{O}_L) \cap \mc{O}_K$, then $a$ is locally represented by $Q$ over $\mc{O}_K$.

In particular, if $(\Lambda,Q)$ is a regular $\mc{O}_K$-lattice (e.g. if its class number is $1$), then ISTR holds for all odd-degree extensions of $K$.
\end{prop}
\begin{proof}
Assume that, on the contrary, there exists a prime $\mf{p}$ of $K$ with $a \not\in Q({\mc{O}_{\mf{p}}})$. There exists a prime $\mf{q}$ of $L$ of odd degree over $\mf{p}$ (as the global degree is the sum of the local degrees, see e.g.~\cite[15:3(1)]{OMe00}); note that this argument works both for archimedean and discrete primes $\mf{p}$. Thus, by \Cref{T:Springer-local}, we get $a \not\in Q(\mc{O}_{\mf{q}})$. Hence, in particular, $a \not\in Q({\mc{O}_L})$, which is a contradiction.

The second part of the statement is clear.
\end{proof}

\section{Finiteness results}
\label{sec:finite}

\begin{thm}\label{L:finitely-many-degree-d-failures}
Let $d \in \zz^+$, $K$ a totally real  number field, $(\Lambda, Q)$ a positive definite quadratic $\mc{O}_K$-lattice, $a \in \mc{O}_K$.
There exist only finitely many totally real number fields $L$ of degree $d$ over $K$ such that $a \in Q({\mc{O}_L})$ but $a \not\in Q({\mc{O}_{L'}})$ for every proper subfield $L'$ of $L$ containing $K$.
\end{thm}
In the proof we shall use the concept of the \emph{house} of a totally real algebraic integer $\alpha$, denoted by $\house{\alpha}$ and defined as $\max_{i=1}^t \lvert \alpha_i \rvert$, where $\alpha_1, \ldots, \alpha_t$ are all the conjugates of $\alpha$.
\begin{proof}
Let $n$ be the rank of $\Lambda$.
There is a basis $\{v_1, \ldots, v_n\}$ of $K\Lambda$ such that  $\Lambda \subseteq \mc{O}_Kv_1 + \cdots + \mc{O}_Kv_n$.
It follows that, for any number field extension $L/K$, $\Lambda_{\mc{O}_L} \subseteq \mc{O}_Lv_1 + \cdots + \mc{O}_Lv_n$.

By \cite[Lemma 4]{KY23} there exists a fixed constant $C$ such that, for every totally real number field extension $L/K$ and for every $\beta_1, \ldots, \beta_n \in \mc{O}_L$, one has $C \cdot \house{Q(\beta_1v_1 + \cdots + \beta_nv_n)} > \max_{i=1}^n \house{\beta_i}^2$.

If now $L$ is a totally real number field containing $K$ such that $a \in Q({\mc{O}_L})$, then there exist $\beta_1, \ldots, \beta_n \in \mc{O}_L$ such that $a = Q(\beta_1v_1 + \cdots + \beta_nv_n)$; and if additionally $a \not\in Q({\mc{O}_{L'}})$ for every proper subfield $L'$, then $L$ must be generated by $\beta_1, \ldots, \beta_n$ over $K$. 
By the previous paragraph, $\beta_1, \ldots, \beta_n$ satisfy $\house{\beta_i}^2 < C\house{a}$.
As there are only finitely many totally real algebraic integers $\beta$ of degree at most $d \cdot [K : \qq]$ and with $\house{\beta}^2 < C\house{a}$ (see, for example, Northcott's Theorem \cite[Theorem 1]{Northcott-Inequality}),
we conclude that only finitely many such fields $L$ may exist.
\end{proof}

In \Cref{O:only-local-failures}, we have seen that the only elements for which ISTR can fail are those which are represented locally but not globally. The following proposition combines this with the almost local-global principle of Hsia, Kitaoka and Kneser to get a finiteness result:

\begin{prop}\label{P:finiteness-HKK}
Let $K$ be a totally real number field, $(\Lambda, Q)$ a positive definite quadratic $\mc{O}_K$-lattice of rank at least $5$.
There exists a finite set $S \subseteq \mc{O}_K^+ \setminus Q({\mc{O}_K})$ with the property that, for every field extension $L/K$ of odd degree, we have
$Q({\mc{O}_L}) \cap \mc{O}_K^+ \subseteq Q({\mc{O}_K}) \cup S\mc{O}_K^{\times 2}$, and if $L$ is not totally real, then we additionally have $(Q({\mc{O}_L}) \cap \mc{O}_K^+) \cup \lbrace 0 \rbrace = Q({\mc{O}_K}) \cup S\mc{O}_K^{\times 2}$.
\end{prop}
\begin{proof}
Let $S$ be a set which consists of exactly one representative from each set of the form $\alpha\mc{O}_K^{\times 2}$ where $\alpha \in \mc{O}_K^+ \setminus Q(\mc{O}_K)$ is locally represented by $Q$ over $\mc{O}_K$. The local-global principle established in \cite[Theorem 3]{HKK} implies that $S$ is finite. We check that such $S$ satisfies the statement: First, clearly $S \subseteq \mc{O}_K^{+} \setminus Q(\mc{O}_K)$. Further, by \Cref{O:only-local-failures}, every element of $Q(\mc{O}_L) \cap \mc{O}_K^{+}$ is locally represented by $Q$ over $K$; thus it either belongs to $Q(\mc{O}_K)$, or it is not globally represented and hence it lies in $S\mc{O}_K^{\times 2}$.
This shows that $Q({\mc{O}_L}) \cap \mc{O}_K^+ \subseteq Q({\mc{O}_K}) \cup S\mc{O}_K^{\times 2}$. 

Now assume additionally that $L$ is not totally real, and pick $a \in S$. By the choice of $S$, we have $a \in Q({\mc{O}_{\mf{p}}})$ for every prime $\mf{p}$ of $K$, and thus also for every prime of $L$. Since $Q$ is indefinite over $\mc{O}_L$, \Cref{T:local-global-indefinite} implies that $a \in Q({\mc{O}_L})$.
From this, it follows that $S\mc{O}_K^{\times 2} \subseteq Q(\mc{O}_L)$.
\end{proof}

\begin{cor}\label{C:finitely-many-degree-d-failures}
Let $K$ be a totally real number field, $(\Lambda, Q)$ a positive definite quadratic $\mc{O}_K$-lattice of rank at least $5$, $d \in \zz^+$.
There exist only finitely many totally real field extensions $L/K$ of degree $2d+1$ such that $Q(\mc{O}_L) \cap \mc{O}_K \neq Q(\mc{O}_K)$, but $Q(\mc{O}_{L'}) \cap \mc{O}_K = Q(\mc{O}_K)$ for every proper subfield $L'$ of $L$ containing $K$.
\end{cor}
\begin{proof}
This follows by combining \Cref{P:finiteness-HKK} and \Cref{L:finitely-many-degree-d-failures}: \Cref{P:finiteness-HKK} gives the finite set $S$ of elements for which ISTR may fail and \Cref{L:finitely-many-degree-d-failures} states that the desired finiteness result holds for each of them separately.
\end{proof}
\begin{rem}\label{rem:rank-3-4}
If $K$ is a totally real number field, then for positive definite quadratic lattices over $\mc{O}_K$ of rank $3$ or $4$, a similar local-global principle as the one used in the proof of \Cref{P:finiteness-HKK} exists for the representation of \textit{square-free} totally positive elements, see, for example, Section 5 in \cite{Schulze}.
Using this, one can state a variation of \Cref{P:finiteness-HKK} and \Cref{C:finitely-many-degree-d-failures} also for lattices of rank $3$ or $4$.
\end{rem}
\begin{rem}
Note that \Cref{C:finitely-many-degree-d-failures} (even \Cref{L:finitely-many-degree-d-failures}) is false when one would allow also extensions $L/K$ where $L$ is not totally real, since in this case, by \Cref{T:local-global-indefinite}, any element of $\mc{O}_K$ that was represented everywhere locally but not globally, becomes represented over $\mc{O}_L$.
See e.g.~\Cref{T:15-representation} for a concrete example.
\end{rem}

\section{Failures of ISTR}\label{sect:examples}
We shall first construct an infinite family of quadratic forms over $\zz$ which represent a certain natural number over the ring of integers of a totally real cubic field, but not over $\zz$.

For an integral domain $R$, $n \in \zz^+$, and $a_1, \ldots, a_n \in R$, we shall denote by $\langle a_1, \ldots, a_n \rangle$ the \emph{diagonal quadratic form} given by $R^n \to R : (x_1, \ldots, x_n) \mapsto \sum_{i=1}^n a_ix_i^2$.
In particular, $\langle a_1, \ldots, a_n \rangle(R)$ denotes the set $Q(R)$ when $Q = \langle a_1, \ldots, a_n \rangle$.
We denote by $\perp$ the orthogonal sum of quadratic lattices.

Let us denote by $K_{49}$ the number field of discriminant $49$.
This is a cyclic cubic extension of $\qq$, in particular, it is totally real.
It can be defined alternatively as the splitting field of the polynomial $T^3 + T^2 - 2T - 1$ over $\qq$, or as the maximal totally real subfield of the seventh cyclotomic field, i.e.\ the field generated over the rationals by $\omega = \zeta_7 + \zeta_7^{-1}$, where $\zeta_7$ is a primitive $7$th root of 1.

\begin{prop}\label{P:infinite-family-Springer-failure}
Let $n, m \in \zz^+$ such that $n \equiv 28 \bmod 32$, $m \equiv 5 \bmod 8$, and $3.25m < n < 4m$.
Then
\begin{enumerate}
\item $n \not\in \langle 1, 1, 1, m \rangle(\zz)$, but
\item $n - m\omega^2 \in \langle 1, 1, 1 \rangle(\mc{O}_{K_{49}})$; in particular $n \in \langle 1, 1, 1, m \rangle(\mc{O}_{K_{49}})$.
\end{enumerate}
\end{prop}
\begin{proof}
First, assume that $n \in \langle 1, 1, 1, m \rangle(\zz)$, i.e.~$n = w^2 + x^2 + y^2 + mz^2$ for certain $w, x, y, z \in \zz$.
By the assumption that $n < 4m$, this is possible only if $z^2 \in \lbrace 0, 1 \rbrace$, since otherwise $n - mz^2$ would be negative.
But this means either $n = w^2 + x^2 + y^2$ or $n - m = w^2 + x^2 + y^2$, and by reducing modulo $32$ (respectively, $8$) one sees that neither option is possible.
We conclude that $n \not\in \langle 1, 1, 1, m \rangle(\zz)$.

We now show that $n - m\omega^2 \in \langle 1, 1, 1 \rangle(\mc{O}_{K_{49}})$.
Note that $\omega^2 \prec 3.25$, which implies that $n - m\omega^2$ is totally positive.
By \cite[Proposition 3.1]{KubaCubicPythagoras} it now suffices to show that $-(n - m\omega^2)$ is not a square in $\zz_2[\omega]$.
Note that an arbitrary element of $\zz_2[\omega]$ can be written as $\sum_{i=0}^{+\infty} u_i 2^i$ with $u_i \in S = \lbrace 0, 1, \omega, \omega^2, 1 + \omega, 1 + \omega^2, \omega + \omega^2, 1 + \omega + \omega^2 \rbrace$.
For the sake of a contradiction, suppose we had $m\omega^2 - n = x^2$ for some $x = \sum_{i=0}^{+\infty} u_i 2^i \in \zz_2[\omega]$ with $u_i \in S$.
By reducing this equation modulo $2$, one sees that $u_0^2 \equiv \omega^2 \bmod 2$ and thus $u_0 = \pm\omega$; replacing $x$ by $-x$ if needed, we may assume $u_0 = \omega$.
But now we compute that
$$\omega^2 + 4(u_1^2 + \omega u_1) \equiv (\omega + 2u_1)^2 \equiv x^2 \equiv m\omega^2 - n \equiv 5\omega^2 + 4 \equiv \omega^2 + 4(1+\omega^2) \bmod 8.$$
We thus need $u_1 \in S$ to satisfy $u_1^2 + \omega u_1 \equiv 1 + \omega^2 \bmod 2$, and one verifies by hand that this holds for none of the $8$ elements of $S$.
\end{proof}
\begin{thm}\label{C:infinite-family-Springer-failure}
Let $k \in \zz^+$, $k \geq 4$.
Then there exists $n \in \zz^+$ such that $\langle 1, 1, 1, 8k+5 \rangle$ represents $n$ over $\mc{O}_{K_{49}}$ but not over $\zz$.
In particular, there exist infinitely many quadratic forms of rank $4$ over $\zz$ which are pairwise non-isometric over $\qq$ and 
for which ISTR fails.
\end{thm}
\begin{proof}
For the first part, observe that, for $k = 4$ and thus $m = 8k+5=37$, one can choose $n = 124$ and conclude with \Cref{P:infinite-family-Springer-failure}.
For $k > 4$, there are more than $32$ natural numbers between $3.25m$ and $4m$, hence in particular some element $n \equiv 28 \bmod 32$; as such one can again apply \Cref{P:infinite-family-Springer-failure}.

For the second part, observe, by comparing determinants, that the forms $\langle 1, 1, 1, 8k+5 \rangle$ and $\langle 1, 1, 1, 8l+5 \rangle$ are not isometric over $\qq$ if $(8k+5)(8l+5)$ is not a square, i.e.~there cannot be a linear change of variables over $\qq$ transforming one form into the other.
And clearly, one can find infinitely many positive integers $k_1, k_2, k_3, \ldots$ such that $(8k_i + 5)(8k_j + 5)$ is never a square for $i \neq j$.
\end{proof}
\begin{rem}
\Cref{C:infinite-family-Springer-failure} implies that, for every $D \in \zz^+$ with $D \geq 4$, there exist infinitely many quadratic forms over $\zz$ of rank $D$ which are pairwise non-isometric over $\qq$ and for which ISTR fails.
It suffices to observe that, if $Q$ is a positive definite form over $\zz$ and $a \in Q({\mc{O}_K}) \setminus Q(\zz)$, then also for $Q' = Q \perp \langle a + 1 \rangle$ we have $a \in Q'({\mc{O}_K}) \setminus Q'(\zz)$.
\end{rem}

\begin{eg}
Consider the setup of \Cref{P:infinite-family-Springer-failure} but replace the form $\langle 1, 1, 1, m \rangle$ by $\langle 1, 1, m \rangle$.
Numerical data suggests that a similar result should hold: for many $m \geq 7$, there seems to be some $n$ slightly larger than $3.25m$ such that $n \not\in \langle 1, 1, m \rangle(\zz)$ but $n- m\omega^2 \in \langle 1, 1 \rangle(\mc{O}_{K_{49}})$. For $m=2p$ and $m=4k+2$, we verified this numerically for the first 200 primes $p$ greater than three and natural numbers $k\le 200$. 

Presumably analytical techniques could be used to show that this happens infinitely often.
\end{eg}
\begin{eg}
We note that ISTR can even fail for binary quadratic forms: the form $\langle 1, 71 \rangle$ does not represent $232$ over $\zz$, but in $\mc{O}_{K_{49}}$ we have $(4\omega^2 - \omega - 16)^2 + 71\omega^2 = 232$ for $\omega = \zeta_7 + \zeta_7^{-1}$.

Furthermore, observe that, when $\alpha, \beta$ are any totally real algebraic integers with $\alpha^2+71\beta^2=232$, then $\frac{232}{71} \succ \beta^2$.
As $\frac{232}{71} < 4$, it follows from Kronecker's Theorem \cite{Kro57} that, up to conjugation, $\beta^2 = 2 + 2\cos(\frac{2\pi}{n})$ for some $n \in \zz^+$.
Since $2+2\cos(\frac{2\pi}{n}) < \frac{232}{71}$ if and only if $n \leq 7$, one obtains
that this forces $\beta^2 \in \lbrace 0, 1, 2, \frac{3+\sqrt{5}}{2}, 3, \omega^2 \rbrace$.
Thus, up to conjugation and switching signs, the only possibilities for the pair $(\alpha, \beta)$ are $(\sqrt{232}, 0)$, $(\sqrt{161}, 1)$, $(\sqrt{90} ,\sqrt{2})$, $\bigl(\sqrt{\frac{1}{2}(251{-}71\sqrt{5})}, \frac{1+\sqrt{5}}{2}\bigr)$, $(\sqrt{19}, \sqrt{3})$, and $( 4\omega^2 - \omega - 16, \omega)$.
In particular, $K_{49}$ is the unique minimal totally real number field of odd degree in which the equation $x^2 + 71y^2 = 232$ has an integral solution.
\end{eg}

Further, we show that ISTR fails fairly often: For a (quite arbitrarily chosen) quadratic form over $\qq$, it fails for the majority of cubic fields with small discriminant.

\begin{eg}
Consider the form $Q=\langle 1,1,14 \rangle$ over $\zz$ and the first $100$ totally real cubic fields, ordered by discriminant.
One can compute that if $K$ is one of these fields and $\mathrm{disc} (K) \neq 43^2$, then ISTR fails for $Q$ over the extension $K/\qq$. For $K$ with discriminant $43^2$, we do not know whether ISTR fails for $Q$.

More precisely: There are $54$ natural numbers under $1500$ which are locally but not globally represented by $Q$ over $\zz$; the smallest is $3$, the largest is $1428$.
These are all represented by $\langle 1,2,7 \rangle$, which is everywhere locally isometric to $\langle 1, 1, 14 \rangle$.
One verifies numerically that, for every totally real cubic extension $K/\qq$ where $\mathrm{disc} (K) \leq 3132$ and $\mathrm{disc} (K) \neq 43^2$, there is at least one of these $54$ exceptional numbers which becomes represented by $Q$ over $\mc{O}_K$. (For the field with $\mathrm{disc} (K)=43^2$, there is no natural number under $4000$ which becomes represented over $\mc{O}_K$.)
In particular, ISTR fails for $Q$ in these $99$ cubic extensions.
\end{eg}

The following is an illustration for a failure of ISTR over a base field other than $\qq$.

\begin{eg}
Let $K = \qq(\sqrt{21})$ and let $L =\qq(\vartheta)$, where $\vartheta=\zeta_{21}+\zeta^{-1}_{21}$, i.e.~$L$ is the maximal totally real subfield of the $21$st cyclotomic field. It is of degree $6$ as $\vartheta$ has minimal polynomial $T^6 - T^5 - 6T^4 + 6T^3 + 8T^2 - 8T + 1$ over $\qq$.
Then $\sqrt{21} = 2\vartheta^5 + 2\vartheta^4 - 10\vartheta^3 - 8\vartheta^2 + 12\vartheta + 3 \in L$, and $L/K$ is an extension of degree $3$.

One verifies that the element $7 + (\frac{1+\sqrt{21}}{2})^2$ is not a sum of $4$ squares in $\mc{O}_K = \zz[\frac{1+\sqrt{21}}{2}]$, but it is a sum of $4$ squares in $\mc{O}_L = \zz[\vartheta]$, since it is equal to
$$ 1^2 + (\vartheta - 1)^2 + (\vartheta^2 + \vartheta - 2)^2 + (\vartheta^5 - \vartheta^4 - 7\vartheta^3 + 5\vartheta^2 + 12\vartheta - 5)^2.$$
\end{eg}

Many other examples could be produced. For example, one can check that for $\alpha=12+2\sqrt{13}$ and the degree 3 field extension $\qq(\zeta_{13}+\zeta_{13}^{-1}) / \qq(\sqrt{13})$, one gets $\alpha \notin \langle 1, 1, 1, 1 \rangle \bigl(\zz\bigl[\frac{1+\sqrt{13}}2\bigr]\bigr)$ but $\alpha \in \langle 1, 1, 1, 1 \rangle \bigl(\zz[\zeta_{13}+\zeta_{13}^{-1}])$.

\smallskip

We conclude with an example illustrating that the fact that a positive definite quadratic form over $\zz$ is not regular (i.e.~there exist elements in $\zz^+$ represented by the form everywhere locally but not over $\zz$) does not imply a failure of ISTR for this form. In other words, the second part of \Cref{O:only-local-failures} cannot be turned into an equivalence.

Consider the form $Q = \langle 1, 2, 5, 5 \rangle$.
It is well known that $Q(\zz) = \zz_{\geq 0} \setminus \lbrace 15 \rbrace$, see e.g.~\cite{Bhargava-15}, in particular, $Q$ represents all natural numbers everywhere locally and thus over $\qq$.
\begin{thm}\label{T:15-representation}
Let $K$ be a number field such that $[K : \qq]$ is odd.
Then $15 \in \langle 1, 2, 5, 5 \rangle(\mc{O}_K)$ if and only if $K$ is not totally real.
\end{thm}
\begin{proof}
First assume that $K$ is totally real, and assume that $15 = w^2 + 2x^2 + 5y^2 + 5z^2$ for $w, x, y, z \in \mc{O}_K$.
We first argue why, in this case, one must have $y, z \not\in \zz$.
Suppose on the contrary that $z \in \zz$. We must then have $z^2 \in \lbrace 0, 1 \rbrace$, otherwise $15 - 5z^2$ would be negative.
In other words, either $15 = w^2 + 2x^2 + 5y^2$ or $10 = w^2 + 2x^2 + 5y^2$.
However, neither $10$ nor $15$ is represented by $\langle 1, 2, 5 \rangle$ over the field $\qq_5$, hence in particular they are not represented by $\langle 1, 2, 5 \rangle$ over $\qq$.
By the classical Springer's Theorem, they cannot be represented over $K$ either.
We obtain a contradiction and conclude that $z \not\in \zz$; by symmetry, also $y \not\in \zz$.

Now set $d = [K : \qq]$ and recall Siegel's result that, for any $\alpha \in \mc{O}^{+}_K$ different from $1$ or $\frac{3 \pm \sqrt{5}}{2}$ we have that $\Tr_{K/\qq}(\alpha) > \frac{3d}{2}$ \cite[Theorem III]{SiegelTrace}, where $\Tr_{K/\qq}$ denotes the trace in the extension $K/\qq$.
Since we excluded the possibility that $y^2, z^2 \in \lbrace 0, 1 \rbrace$ in the previous paragraph, and $\frac{3 \pm \sqrt{5}}{2}$ does not lie in $K$ since $d$ is odd, we must have $\Tr_{K/\qq}(y^2), \Tr_{K/\qq}(z^2) > \frac{3d}{2}$.
We compute that
\begin{align*}
15d &= \Tr_{K/\qq}(w^2 + 2x^2 + 5y^2 + 5z^2) = \Tr_{K/\qq}(w^2) + 2\Tr_{K/\qq}(x^2) + 5\Tr_{K/\qq}(y^2) + 5\Tr_{K/\qq}(z^2) \\
&> 0 + 0 + 5d\frac{3}{2} + 5d\frac{3}{2} = 15d,
\end{align*}
which is impossible.

Now, assume that $K$ is not totally real.
Since $15$ is everywhere locally represented by $\langle 1, 2, 5, 5 \rangle$, and $Q_{\mc{O}_K}$ is indefinite, by \Cref{T:local-global-indefinite} we have $15 \in \langle 1, 2, 5, 5 \rangle(\mc{O}_K)$.
\end{proof}

The following example is in similar spirit: We show that sometimes the failure of ISTR for a given element and a given quadratic form is rare. Note that the exceptional cubic field in this case is not $K_{49}$ which served us in many of our previous examples; it is not even a Galois extension of $\qq$:

\begin{eg}
Consider the form $Q(W,X,Y,Z) = 29W^2+X^2+2Y^2+4Z^2+XZ+YZ$. One has $145 \notin Q(\zz)$; in fact, it is one of the \enquote{escalator lattices} occurring in the proof of the 290-Theorem of Bhargava and Hanke \cite{Bhargava-Hanke}. For a totally real cubic field $K$, we show that $145 \in Q(\mc{O}_K)$ if and only if $K$ is the field with discriminant $229$.

First of all, let $\beta$ be a root of $T^3-4T-1$. Then $\qq(\beta)$ has discriminant $229$ and one can check that $145 \in Q(\zz[\beta])$ by plugging in
\[
W=\beta,\quad X=\beta^2+2\beta-5,\quad Y=2\beta^2-8,\quad Z = -\beta+1.
\]

On the other hand, assume now that $K$ is a totally real cubic field and $145 = 29w^2+x^2+2y^2+4z^2+xz+yz$. This immediately yields $\Tr_{K/\qq}(w^2) \leq 3 \cdot \frac{145}{29} = 15$. First, just as in the previous theorem, we show that $w^2 \notin \zz$: If it were, then $w^2 \in \{0,1,4\}$, so the ternary form $X^2+2Y^2+4Z^2+XZ+YZ$ would represent one of the numbers $145$, $145-29 = 4\cdot29$ and $145-4\cdot 29 = 29$ over $K$ and thus, by the classical Springer's theorem, over $\qq$. However, this is impossible as none of them is represented over $\qq_{29}$.

Thus, $w \notin \qq$, so we have $K=\qq(w)$. This in particular means that for the discriminant $\Delta(w)$ of $\zz[w]$ we have $\mathrm{disc}(K) \leq \Delta(w)$. On the other hand, the inequality \cite[Prop.~2]{KalaSmallRank}, which is a corollary of Schur's bound, yields for cubic fields
 $\Delta(w) \leq \frac{1}{2} \bigl(\Tr_{K/\qq}(w^2)\bigr)^3$. Putting all the inequalities together, we get
\[
\mathrm{disc}(K) \leq \Delta(w) \leq \frac{1}{2} \bigl(\Tr_{K/\qq}(w^2)\bigr)^3 \leq \frac{15^3}{2} = 1687.5.
\]
There are only $50$ totally real cubic fields that satisfy this bound, and for them we checked the representation numerically.
\end{eg}

We leave open several questions for further investigation.
\begin{quess} \ 
\begin{enumerate}[(1)]
\item Do there exist a number field $K$, a quadratic form $Q$ over $\mc{O}_K$, $a \in \mc{O}_K$, and an infinite collection of totally real number field extensions $L/K$ of odd degree such that $a \in Q(\mc{O}_L)$, but $a \not\in Q(\mc{O}_{L'})$ for every proper subfield $L'$ of $L$ containing $K$?
\item Given an odd-degree extension $L/K$ of totally real fields, does there always exist a quadratic form $Q$ over $\mc{O}_K$ and $a \in \mc{O}_K \cap Q(\mc{O}_L) \setminus Q(\mc{O}_K)$?
\end{enumerate}
\end{quess}
\bibliographystyle{alpha}
\bibliography{references}
\end{document}